\documentclass[a4paper]{article}

\usepackage{a4wide}
\usepackage[latin1]{inputenc}
\usepackage[T1]{fontenc}
\usepackage{amsfonts}
\usepackage{amssymb}
\usepackage{amsmath}
\usepackage{amsthm}
\usepackage{mathrsfs}
\usepackage{bbm}
\usepackage{dsfont}

\newcommand{\EE}{\mathds{E}}

\newcommand{\PP}{\mathds{P}}

\newcommand{\R}{\mathds{R}}

\renewcommand{\epsilon}{\varepsilon}

\newtheorem{thm}{Theorem}[section]
\newtheorem{lemme}[thm]{Lemma}
\newtheorem{prop}[thm]{Proposition}

\theoremstyle{definition}

\numberwithin{equation}{section}

\author{Hendrik Weber}
\title{On the short time asymptotic of the  stochastic Allen-Cahn equation}

\begin{document}
\maketitle
\begin{abstract}
A description of the short time behavior of solutions of the Allen-Cahn equation with a smoothened additive noise is presented. The key result is that in the sharp interface limit solutions move according to motion by mean curvature with an additional stochastic forcing. This extends a similar result of Funaki \cite{Fu99} in spatial dimension $n=2$ to arbitrary dimensions.
\end{abstract}

\section{Introduction and main result}
{\em 1. Setting and main result:} For a small parameter $\varepsilon >0$ consider the following stochastic Allen-Cahn equation in an open domain $D$ in $\R^n$ for some $n \geq 2$:
\begin{align}\label{SAC}
 \nonumber \frac{\partial}{\partial t}u^\epsilon (x,t)&=\Delta u^\epsilon(x,t)+\epsilon^{-2} f(u(x,t)) +\epsilon^{-1} \xi^\epsilon(t)  \qquad &&(x,t) \in D \times [0, \infty)\\
\frac{\partial}{\partial \nu} u^\epsilon(x,t)&=0 \qquad &&x \in \partial D\\
\nonumber u^\epsilon(x,0)&= u_0^\epsilon(x) \qquad &&x \in  D.
\end{align}
Here $f(u) = -F'(u)$ is the negative derivative of a symmetric double-well potential. For fixing ideas, assume that $F(u)= \frac{(u^2-1)^2}{4}$ and $f(u)=u-u^3$. In particular $F$ has two global minima at $\pm 1$ and solutions of the dynamical system $\dot{x}=f(x)$, that start outside of zero, converge to one of these minima.  The expression $\xi^\epsilon(t)$ denotes a noise term defined on a probability space $(\Omega,\mathcal{F},\PP)$. The noise $\xi^\epsilon(t)$ is constant in space and smooth in time. For $\epsilon\downarrow 0$ the correlation length goes to zero at a precise rate and $\int_0^t \xi^\epsilon(s)ds$ converges to a Brownian motion pathwisely. The details of the construction and further properties can be found below. 
\vskip2ex
We study the short time evolution of developed surfaces for (\ref{SAC}). More precisely let $\Sigma_0$ the boundary of a set $U_0$ be compactly embedded in $D$ of class $C^{2, \alpha}$ for some $\alpha>0$. Assume that the initial configuration $u^\epsilon(x,t)$ is close to $-1$ on $U_0$ and close to $+1$ on $D \setminus U_0$ with a transition layer of order $O(\epsilon)$. We show that for short times there exist two phases and the evolution of the phase boundary follows two influences - the tendency to minimize the boundary and a stochastic effect. The main result is:
\begin{thm}\label{MR}
Consider the problem (\ref{SAC}) with the noise term $\xi^\epsilon(t)$ as constructed below. In particular suppose that the approximation rate $\gamma$ verifies $\gamma < \frac{2}{3}$. Then for any compactly embedded hypersurface $\Sigma_0=\partial U_0$ of class $C^{2,\alpha}$ there exist initial conditions $u^\epsilon$, a positive stopping time $\tau$ and randomly evolving closed hypersurfaces $(\Sigma(t))_{0 \leq t \leq \tau}$ such that the following hold:
\begin{itemize}
\item[(i)] The surfaces $(\Sigma_t)_{0 \leq t \leq \tau}$ evolve according to stochastically perturbed motion by mean curvature, e.g. the normal velocity $V$ at each point is given by
\[
  V= (n-1)\kappa  - c_0 \dot{W(t)}.
\]
\item[(ii)] $\sup_{0 \leq t \leq \tau} \|u^\epsilon(x,t)- \chi_{\Sigma_t} \|_{L^2(D)} \rightarrow 0$ almost surely as $\epsilon$ goes to zero.
 \end{itemize}
\end{thm}
Here $\kappa$ denotes the mean curvature of the surface at a given point. The constant $c_0$ is given by
\[
 c_0=\frac{\sqrt{2}}{\int^1_{-1} \sqrt{F(u)} \mathrm{du}}.
\]
The function $\chi_{\Sigma_t}$ is a step function taking the value $-1$ in the interior and $+1$ on the exterior. The precise meaning of the geometric evolution will be given in the next section. 
\vskip2ex
The noise scaling $\epsilon^{-1} \xi^\epsilon(t)$ can be interpreted as follows: Consider the stochastic equation
\begin{equation}\label{SAC2}
 \frac{\partial v}{\partial t}= \Delta v+ f(v)+ \epsilon \xi^\epsilon(t).
\end{equation}
Equation (\ref{SAC}) can be obtained from this equation by diffusive scaling: $u(x,t)=v(\epsilon^{-1} x, \epsilon^{-2} t)$. The intuition is that in (\ref{SAC2}) surfaces should move with velocity $V=(n-1)\kappa + c(\epsilon \xi^\epsilon(t))$. Here $c$ is the speed of a travelling wave solution corresponding to a perturbation of the potential through $\epsilon \xi^\epsilon(t)$. Then after rescaling one obtains as normal velocity $V=\kappa + \epsilon^{-2} \times \epsilon^{1} c(\epsilon \xi^\epsilon(t))$ such that the random term becomes a quantity of order $O(1)$. The significant observation is that the noise term does not rescale. Actually this observation is characteristic for our result. Even in the limit the Brownian motion can be considered pathwise and there is nowhere any need to work with stochastic integrals.
\vskip2ex

{\em 2. The white noise approximation:} Let $(W(t),t \geq 0)$ be a Brownian motion defined on a probability space $(\Omega,\mathcal{F},\PP)$. For technical reasons extend the definition of $(W(t),t \geq 0)$ to negative times by considering an independent Brownian motion $(\widetilde{W}(t),t \geq 0)$ and setting $W(t)=\widetilde{W}(-t)$ for $t < 0$. Then $(W(t),t \in \R)$ is a gaussian process with independent stationary increments and a distinguished point $W(0)=0$ a.s. Let $\rho$ be a mollifying kernel i.e. $\rho\colon \mathds{R} \to \mathds{R}_+$ is smooth and symmetric with $\rho(x)=0$ outside of $[-1,1]$ and $\int \rho(x)\mathrm{dx}=1$. For $\gamma >0$ set $\rho^{\epsilon}(x)=\epsilon^{-\gamma}\rho(\frac{x}{\epsilon^\gamma})$. Then the approximated Brownian motion $W^{\epsilon}(t)$ is defined as usual as
\[
 W^\epsilon(t)=W \ast \rho^{\epsilon}(t)=\int_{-\infty}^\infty \rho^\epsilon(t-s)W(s) \mathrm{ds}.
\]
Note that it is only here that the Brownian motion at negative times is needed. So actually only negative times in $(-\epsilon^\gamma,0]$ will play a role. The parameter $\gamma$ determines how quickly the approximations converge to the true integrated white noise. We will always assume 
\[
\gamma < \frac{2}{3}
\]
 in order to have the needed pathwise bounds on the white noise approximations.
\begin{prop}
 Let $\xi^\epsilon(t)=\dot{W}^\epsilon(t)$ denote the derivative of $W^\epsilon$. Then the following properties hold:
\begin{itemize}
 \item[(i)] $\xi^\epsilon(t)$ is a stationary centered gaussian process with $\EE[\xi^\epsilon(t)^2]= \epsilon^{-\gamma} |\rho|_{L^2}^2$. 
 \item[(ii)] The correlation length of $\xi^\epsilon(t)$ is $2 \epsilon^\gamma$ i.e. if $|s-t| \geq 2 \epsilon^\gamma$ then  $\xi^\epsilon(t)$  and  $\xi^\epsilon(s)$ are independent. 
 \item[(iii)] If $\gamma < \tilde{\gamma}$ for every positive time $T$ there exists a non-random constant $C$ such that
\[
 \PP\Bigl[ \exists \epsilon_0  \text{ s.th. }\forall \epsilon \leq \epsilon_0 \sup_{0 \leq t \leq T}|\xi^\epsilon(t)| \leq C \epsilon^{-\frac{\tilde{\gamma}}{2}} \Bigr] =1.
\]
In particular for $\gamma < \frac{2}{3}$ for $\epsilon$ small enough 
\begin{equation}\label{noi1}
 \xi^\epsilon(t) \leq C\epsilon^{ -\frac{1}{3}}.
\end{equation}

\end{itemize}
\end{prop}
\begin{proof}
One can write
\[
 \xi^\epsilon(t)= \int_{-\infty}^\infty \frac{\mathrm{d}}{\mathrm{dt}}\rho^\epsilon(t-s)W(s) \mathrm{ds}= \int_{-\infty}^\infty \rho^\epsilon(t-s)\mathrm{d}W(s) \quad \text{a.s.},
\]
where the first equality follows from differentiating under the integral and the second from stochastic integration by parts. Then properties (i) and (ii) follow from standart properties of the stochastic integral. To see (iii) write
\begin{align*}
  |\xi^\epsilon(t)|&= \Bigl|\int_{t-\epsilon^\gamma}^{t+\epsilon^\gamma}\epsilon^{-2 \gamma} \rho'\left(\frac{t-s}{\epsilon^\gamma}\right)W(s) \mathrm{ds} \Bigr|\\
&\leq \Bigl| \epsilon^{-2 \gamma} \int_{t-\epsilon^\gamma}^{t+\epsilon^\gamma}\rho'\left(\frac{t-s}{\epsilon^\gamma}\right)W(t) \mathrm{ds} \Bigr|+ \Bigl|\epsilon^{-2 \gamma}  \int_{t-\epsilon^\gamma}^{t+\epsilon^\gamma}\rho'\left(\frac{t-s}{\epsilon^\gamma}\right)(W(t)-W(s)) \mathrm{ds} \Bigr|.
\end{align*}
The first term vanishes due to $\int_{t-\epsilon^\gamma}^{t+\epsilon^\gamma}\rho'\left(\frac{t-s}{\epsilon^\gamma}\right)ds=0$. One obtains
\begin{align*}
  |\xi^\epsilon(t)| &\leq \Bigl|\epsilon^{-2 \gamma}  \int_{t-\epsilon^\gamma}^{t+\epsilon^\gamma}\rho'\left(\frac{t-s}{\epsilon^\gamma}\right)(W(t)-W(s)) \mathrm{ds} \Bigr|\\
&\leq \Bigl|\epsilon^{-2 \mathrm{\gamma}}  2 \epsilon^\gamma \| \rho'\|_\infty \text{osc}_{s \in [t-\epsilon^\gamma,t+\epsilon^\gamma]}W(s)  \Bigr|.
\end{align*}
The oscillation is defined as $\text{osc}_{s \in [t-\epsilon^\gamma,t+\epsilon^\gamma]}W(s):=\sup_{s \in [t-\epsilon^\gamma,t+\epsilon^\gamma]} W(s)-\inf_{s \in [t-\epsilon^\gamma,t+\epsilon^\gamma]} W(s)$. 

\vskip2ex

Now one can apply L\'evy's well known result on the modulus of continuity of  Brownian paths (See e.g. \cite{KS91} Theorem 9.25 on page 114):
\[
 \PP \Bigl[\limsup_{\delta\rightarrow 0} \frac{1}{g(\delta)} \underset{t-s \leq \delta} {\max_{0 \leq s < t \leq T}} |W(t)-W(s)|=1  \Bigr]=1,
\]
where the modulus of continuity is given by $g(\delta)= \sqrt{2 \delta \log(\frac{1}{\delta})}$. In particular there exists almost surely a (random!) $\epsilon_0$ such that for $\epsilon \leq \epsilon_0$ we have $\sup_{t \in [0,T]} \text{osc} _{s \in [t-\epsilon^\gamma, t+\epsilon^\gamma]} W(s) \leq (2\epsilon^{\gamma})^{\frac{1}{2}-\frac{\tilde{\gamma}-\gamma}{2\gamma}}$. This gives the desired estimate
\[
 |\xi^\epsilon(t)| \leq \epsilon^{-\gamma} 2\|\rho' \|_{\infty}(2\epsilon^{\gamma})^{\frac{1}{2}-\frac{\tilde{\gamma}-\gamma}{2\gamma}}\leq C \epsilon^{-\tilde{\gamma}/2}.
\]
\end{proof}
We will need a similar bound on the derivatives of $\xi^\epsilon$
\begin{prop}
Consider  the process $\dot{\xi}^\epsilon(t)$. Then if $\gamma < \tilde{\gamma}$  for every positive time $T$ there exists a constant $C$ such that
\[
 \PP\Bigl[ \exists \epsilon_0 \quad \forall \epsilon \leq \epsilon_0 \sup_{0 \leq t \leq T}|\dot{\xi}^\epsilon(t)| \leq C\epsilon^{-\frac{3\tilde{\gamma}}{2}} \Bigr] =1.
\]
In particular for $\gamma< \frac{2}{3}$ and $\epsilon$ small enough 
\begin{equation}\label{noi2}
|\dot{\xi}^\epsilon(t)| \leq C\epsilon^{-1}. 
\end{equation}
\end{prop}
\begin{proof}
The proof is similar to the one above:
\begin{align*}
|\dot{\xi}^\epsilon(t)| &\leq \Bigl|\int_{t-\epsilon^\gamma}^{t+\epsilon^\gamma}\epsilon^{-3 \gamma} \rho''\left(\frac{t-s}{\epsilon^\gamma}\right)W(s) \mathrm{ds} \Bigr|\\
&\leq \Bigl|\int_{t-\epsilon^\gamma}^{t+\epsilon^\gamma}\epsilon^{-3 \gamma} \rho''\left(\frac{t-s}{\epsilon^\gamma}\right)(W(t)-W(s)) \mathrm{ds} \Bigr|\\
&\leq 2 \epsilon^{-2 \gamma} \| \rho'' \|_\infty \text{osc}_{s \in [t-\epsilon^\gamma,t+\epsilon^\gamma]}W(s):
\end{align*}
Then one applies L\'evy's modulus of continuity again to see that almost surely for $\epsilon\leq \epsilon_0(\omega)$ one has $\text{osc}_{s \in [a,b]}W(s)\leq (2\epsilon^{\gamma})^{\frac{1}{2}-\frac{3\tilde{\gamma}-3\gamma}{2\gamma}}$ and obtains the desired result:
\[
 |\dot{\xi}^\epsilon(t)| \leq 2 \epsilon^{-2 \gamma} \| \rho'' \|_\infty (2\epsilon^{\gamma})^{\frac{1}{2}-\frac{\tilde{3\gamma}-3\gamma}{2\gamma}} =C \epsilon^{\frac{3\tilde{\gamma}}{2}}.
\]
\end{proof}

\vskip2ex

{\em 3. Motivation and related works:} Solutions of the Allen-Cahn equation 
\[
 \frac{\partial u}{\partial t}=\Delta u  + \frac{1}{\epsilon^2}f(u)
\]
evolve according to the $L^2$ gradient flow of the real Ginzburg-Landau energy functional: 
\[
 \mathcal{H}^\epsilon(u)= \int | \nabla u|^2+ \frac{1}{\epsilon^2}F(u).
\]
There are two different effects. The reaction term $\epsilon^{-2}f(u)$ pushes solutions to the two minima $\pm 1$ and the diffusion term $\Delta u$ tends to smoothen the solution. For small $\epsilon$ there will be two phases, corresponding to regions where the solution is close to $\pm 1$. The width of the transition layer between those two phases is of the order $O(\epsilon)$. Then the evolution gradually shrinks the transition layer.  
\vskip2ex
This behavior is the motivation to consider the Allen-Cahn equation as a simple model of a two phase system which is driven by the surface energy without conservation of mass. Allen and Cahn \cite{AC79} introduced it to model the interface motion between different cristaline structures in alloys.  In the deterministic setting there were major advances in connection with the improved understanding of the theory of geometric flows of surfaces as initiated for example by \cite{ES91, CGG91} in the early nineties. In particular in \cite{ESS92} it was shown that in the limit $\epsilon \downarrow 0$ solutions only attain the values $\pm 1$ and the phase boundary evolves according to motion by mean curvature. The key difficulty here is to find a description of the geometric evolution which is global in time. A similar result for short times was established in \cite{MS95}.
\vskip2ex
Stochastic perturbations of this effect have also been considered. From a modelling point of view an additional noise term can account for inaccuracies of the simplified model or as effects of thermal perturbations. From a mathematical point of view it is a very interesting and challenging question to study stochastically perturbed evolutions of surfaces and the Allen-Cahn setup is one possible point of view. In \cite{Fu95} Funaki considered the case of the Allen-Cahn equation in one space dimension with a space-time white noise. He showed that in the limit $\epsilon \downarrow 0$ on the right time-scale solutions only attain values $\pm 1$ and the boundary point essentially performs a Brownian motion. In \cite{Fu99} he studies the two dimensional case with a smoothened noise and shows that for short times solutions evolve according to a stochastically perturbed motion by mean curvature. His analysis relies on a comparison theorem which requires the noise to be smooth and a very subtle analysis of a quasi-linear stochastic PDE which describes the boundary evolution. On the level of stochastic surface evolution there were advances by Yip \cite{Y98} and Dirr, Luckhaus and Novaga \cite{DLN01} but a fully satisfactory description is not yet available. Some results based on a stochastic version of the concept of viscosity solutions were announced in \cite{LS98}. Recently the model has enjoyed an increasing interest in the numerical analysis community. For example in \cite{KKL07} numerical approximations of the one-dimensional equation are studied. Numerical analysis of this equation is challenging because all the interesting dynamics happen on a very thin layer which requires to develop adaptive methods which work in the stochastic setting.
\vskip2ex
Our result is a generalization of Funaki's result to arbitrary dimension. We use the same comparison technique to study the equation. Therefore we also need to assume a smoothened noise with correlation length going to zero as $\epsilon$ goes to zero. The description of the surface and the convergence result is based on \cite{DLN01} and fully avoids Funaki's result of weak convergence. In fact this is also a strictly pathwise  result so that all results hold almost surely. 
\vskip2ex
{\em 4. Structure of the paper:}  In Section 2 the technique of \cite{DLN01} to describe motion by mean curvature is briefly reviewed and the main results are stated. In Section 3 the results about the geometric flow are used to proof the behavior of the Allen-Cahn equation.\\
\vskip2ex
{\bf Acknowledgement:} The author expresses his sincere gratitude to Tadahisa Funaki for the great hospitality he received at the University of Tokyo. He also thanks the referee for careful reading and various suggestions.

\section{Stochastic motion by mean curvature}\label{defmbmc}
This section reviews the description of a stochastically perturbed motion by mean curvature given in \cite{DLN01}. A short time existence result for surfaces moving with normal velocity $\mathrm{d} V= (n-1) \kappa \mathrm{dt} + c \mathrm{d}W(t)$, where $\kappa$ denotes the mean curvature, and a pathwise stability result under approximations of the integrated noise are given.
\vskip2ex
Motivated by \cite{ES92} consider the following system
\begin{align}\label{SMBMC1}
\nonumber \mathrm{d} d(x,t)&=g(D^2d(x,t), d(x,t)) \mathrm{dt}+ \mathrm{d} W(t) \qquad &&(x,t) \in \mathcal{O} \times (0,T)\\
|\nabla d|^2&=1 \qquad &&(x,t) \in \partial \mathcal{O} \times (0,T)\\
\nonumber d(x,0)&=d_0(x) \qquad && x \in \mathcal{O},   
\end{align}
on some open bounded domain $\mathcal{O}$. Here $D^2 d$ denotes the Hessian of $d$ and $g(A,q)= \text{tr} ( A(I-qA)^{-1})$ for a symmetric matrix $A$ and $q \in \R$. The initial condition $d_0$ is supposed to be of class $C^{2, \alpha}$ and to verify $|\nabla d|=1$ in $\mathcal{O}$. Furthermore it is assumed that $\nabla d$ is nowhere tangent to the boundary.
\vskip2ex
In order to solve the above system consider $q(x,t)= d(x,t)-W(t)$. Then $q$ solves the system 
\begin{align}\label{SMBMC2}
\nonumber \mathrm{d}q(x,t) &=g(D^2q(x,t), q(x,t)+W(t))\mathrm{dt} \qquad &&(x,t) \in \mathcal{O} \times (0,T)\\
|\nabla q|^2&=1 \qquad &&(x,t) \in \partial \mathcal{O} \times (0,T)\\
\nonumber q(x,0)&=d_0(x) \qquad && x \in \mathcal{O}.   
\end{align}
Due to maximal regularity of the linearized system (\cite{L95}) and a fix point argument the following results are obtained:
\begin{thm}\label{Thm2.1}{\em (\cite{DLN01} Section 4)}
Let $t \mapsto W(t)$ be $\alpha$-H\"older continuous for some $\alpha \in (0,1)$. Then there exists a time $T$ depending only on the $C^{\alpha/2}$-norm of $W$ and the $C^{2,\alpha}$-norm of $d_0$ such that on $\mathcal{O} \times [0,T]$ system (\ref{SMBMC2}) and therefore also (\ref{SMBMC1})  admit a unique solution of class $C^{1 + \alpha/2, 2+ \alpha}$. Moreover if $t\mapsto \tilde{W}(t)$ is another function of class $C^{\alpha}$ and $\tilde{q}$ is the solution (\ref{SMBMC2}) with $W$ replaced by $\tilde{W}$with interval of existence $[0, \tilde{T}]$  on has
\begin{equation}\label{stab1}
\sup_{t \in [0, \min \{ \tilde{T},T \}]} \|q(t,\cdot) - \tilde{q}(t,\cdot)\|_{C^{2, \alpha}} \leq C \|W - \tilde{W} \|_{C^{\alpha/2}([0, \min \{ \tilde{T},T \}])} .
\end{equation}
\end{thm}
Now let $\Sigma_0= \partial U_0$ be as above. In particular $\Sigma_0$ is assumed to be of class $C^{2,\alpha}$. Define the {\em signed distance function} $d_0$ and the indicator $\chi_{\Sigma_0}$ as 
\[
d_0(x)=
\begin{cases}
-\text{dist}(x,\Sigma_0) \qquad &\text{for } x \in U_0\\  
\text{dist}(x,\Sigma_0) \qquad &\text{for } x \in D\setminus U_0
\end{cases}
\]
and
\[
 \chi_{\Sigma_0}(x)= 
\begin{cases}
-1 \qquad &\text{for } x \in U_0\\  
1 \qquad &\text{for } x \in D\setminus U_0.
\end{cases}
\]  
There exists an open environment $\mathcal{O}$ of $\Sigma_0$ such that on $\mathcal{O}$ the function $d_0(x)$ is of class $C^{2,\alpha}$ and $\nabla d$ is nowhere tangent to $\partial \mathcal{O}$. Furthermore on $\mathcal{O}$ it holds $|\nabla d_0|=1$. Then for a given stochastic noise $W(t)$ consider the pathwise solution $d(x,t)$ of (\ref{SMBMC1}) with initial condition $d_0$ on $[0,T(\omega)]$. Define the evolving surfaces $(\Sigma(t), 0 \leq t \leq T(\omega))$ as the zero level sets of $d(x,t)$. Then the following holds:
\begin{thm}{\em ( \cite{DLN01} Section 4)}   
\begin{itemize}
\item[(i)] For every $t$ the function $x \mapsto d(x,t)$ is the signed distance function of $\Sigma(t)$ on $\mathcal{O}$.
\item[(ii)] If $X(0)$ in $\Sigma(0)$. Then up to a stopping time there exists a solution $X(t)$ to the stochastic differential equation
\[
dX(t)= (n-1) \nu(X(t),t) \kappa(X(t),t) \mathrm{d}t + \nu(X(t),t) \mathrm{d}W(t),
\]
with $X(t) \in \Sigma(t)$ almost surely. 
\end{itemize}
\end{thm}
Here $\nu(x,t)$ denotes the exterior normal vector to $\Sigma(t)$ for $x \in \Sigma(t)$. The last observation justifies to say that the surfaces $\Sigma(t)$ evolve according to stochastic motion by mean curvature. Note that we use the convention that $\kappa=\frac{1}{n-1}\sum_{i=1}^{n-1} \kappa_i$ with the principal curvatures $\kappa_i$ such that the factor $(n-1)$ appears which is not present in \cite{DLN01}.

\section{Construction of sub- and supersolutions}\label{CSS}
In this section the link between the boundary dynamic and the Allen-Cahn equation is established. For a related calculation see \cite{Fu99,CHL97}.
\vskip2ex

In order to construct sub- and supersolutions to (\ref{SAC}) consider the following modification of the reaction term: $f(u, \delta)= f(u)+\delta$. The implicit function theorem implies that there exists an interval $[-\tilde{\delta}_0, \tilde{\delta}_0]$ such that for $\delta \in [-\tilde{\delta}_0, \tilde{\delta}_0]$ there exist two solutions $m_{\pm}(\delta)$ of the equation $f(u,\delta)=0$ which are close to $\pm 1$ and that the mappings $\delta \mapsto m_{\pm}(\delta)$ are smooth. Consider the following auxiliary one dimensional problem 
\begin{align}\label{aux}
 \frac{\partial}{\partial t} u(x,t)&= \frac{\partial^2}{\partial x^2} u(x,t) +f(u(x,t))+ \delta \\
\nonumber u(\pm \infty) &= m_{\pm}(\delta).
\end{align}
A travelling wave solution to (\ref{aux}) is a solution $u(x,t)=m(x -c t)$ with a fixed wavespeed $c$. Finding such a solution is equivalent to finding an appropriate waveshape $m(x,\delta)$ and wavespeed $c(\delta)$ such that
\begin{align}\label{TW}
 m''(x)+c(\delta)m'(x)+ \{ f(m(x))+\delta \}&=0\\
\nonumber m(\pm \infty) &= m_{\pm}(\delta).
\end{align}
 The following properties hold:
\begin{lemme}\label{CHL}{\em(\cite{CHL97} Lemma 3.3)}
There exists a constant $\delta_0$  such that for $\delta \in [-\delta_0,\delta_0]$ problem (\ref{TW}) admits a solution $(m(x, \delta), c(\delta))$  where $m$ is increasing in $x$ and this solution is unique up to translation. Furthermore $m$ can be chosen smooth in $\delta$. There exist constants $A$ and $\beta$ such that the following properties hold:
\begin{itemize}
 \item[(i)] $0 < \partial_x m(x,\delta) \leq A$  for all  $(x,\delta) \in \R \times [ -\delta_0,\delta_0]$.
 \item[(ii)] $|\partial_x m(\pm x,\delta)|+|(\partial_x)^2 m(\pm x,\delta)|+| m(\pm x,\delta) -m_{\pm}(\delta)|\leq A e^{-\beta x}$ for all  $(x,\delta) \in \R_+ \times [ -\delta_0,\delta_0]$.
 \item[(iii)] The traveling wave velocity $c(\delta)$ is smooth in $[-\delta_0,\delta_0]$ and $c(0)=0$.
\end{itemize}
\end{lemme}
Actually as pointed out in \cite{Fu99}  $ \partial_\delta c(0)=-c_0=-\frac{\sqrt{2}}{\int^1_{-1} \sqrt{F(u)} \mathrm{du}}$.
\vskip2ex
The idea of the construction is the following: We expect the surface to evolve according to two influences - the surface tension and the stochastic perturbation of the potential making one of the stable states more attractive. Close to the surface the solution should look like a travelling wave interface which is moving with velocity $c(\epsilon \xi)$. This means that solution should behave like
\[
 u(x,t) \approx m\left(\frac{d(x,t)}{\epsilon}, \epsilon \xi^\epsilon(t) \right),
\]
where $d$ is the signed distance function of a surface moving with normal velocity $ V= (n-1)\kappa + \epsilon^{-1}c(\epsilon \xi^\epsilon)$. The standard way of making this idea rigorous is to modify it in such a way that such an approximate solution is a true sub/supersolution and show that the difference between the two cases evolves on a slower time scale than the original dynamic.
\vskip2ex
Fix some initial surface $\Sigma_0$ as in Theorem \ref{MR}. As $\Sigma_0$ is compactly embedded one can fix an $N$ such that all the principle curvatures of $\Sigma_0$ are bounded by $N$. As in Section \ref{defmbmc} one can define a random evolution $(\Sigma^{\pm,\epsilon}(t), 0 \leq t \leq T_{\epsilon,N}^{\pm})$ evolving with normal velocity
\[
 V=(n-1)\kappa +\epsilon^{-1} c(\epsilon \xi^\epsilon(t) \pm \epsilon^\beta).
\]
Here the stopping time $T_{\epsilon,N}^{\pm}$ is defined as the largest time such that the evolution is well defined and such that on $[0, T_{\epsilon,N}^{\pm}]$ the principle curvatures remain bounded by $N$. The constant $\beta$ can be chosen such that $1< \beta < 2$. The condition $\beta > 1$ ensures that in the original time scale the extra term does not have an effect and the condition $\beta < 2$ ensures that the effect is strong enough for the solution to remain a sub/supersolution. Furthermore assume (by shortening the time interval if necessary) that there exists an open set $\mathcal{O}$ such that for  all $ t \in [0, T_{\epsilon,N}^{\pm}]$ the $\eta$-neighborhood of $\Sigma^{\pm}(t)$ is contained in $\mathcal{O}$ for some small $\eta$. Then one can extend the signed distance functions $d^{\pm}(x,t)$ to a smooth function $\tilde{d}^{\pm}$ on all of $[0,T(\omega)] \times D$ such that on  $U(t) \setminus \mathcal{O}$ the function $\tilde{d}^{\pm}$ is smaller than $-\eta$ and on $D\setminus (U(t) \cup \mathcal{O})$ it is larger than $\eta$, such that $|\nabla \tilde{d}^{\pm}| \leq 1$ and such that $\tilde{d}$ is constant close to $\partial D$.
\vskip2ex
Define
\[
 u^{\pm}(x,t) = m\left(\frac{\tilde{d}^{\pm}(x,t)\pm \epsilon^a e^{c_1 t}}{\epsilon}, \epsilon \xi^\epsilon(t) \pm \epsilon^\beta \right),
\]
where $a$ and $c_1$ are constants that will be chosen below. One gets the following conclusion:
\begin{lemme}\label{subsup}
 If one chooses $a$ and $c_1$ properly, there exists a (random) $\epsilon_0>0$ such that for all $\epsilon\leq \epsilon_0$ and for $0 \leq t \leq T^{\pm,\epsilon}_N$ 
\[
 u^{\epsilon,-}(x,t) \leq u^{\epsilon} (x,t) \leq u^{\epsilon,+}(x,t),
\]
for every solution $u^{\epsilon}(x,t)$ of (\ref{SAC}) with initial data verifying  $u^{\epsilon,-}(x,0) \leq u^{\epsilon} (x,0) \leq u^{\epsilon,+}(x,0)$.
\end{lemme}
\begin{proof}
The conclusion will follow by a PDE-comparison principle. We only show the inequality involving $u^+$ the other one being similar. Let us calculate
\begin{align*}
\partial_t u^{\epsilon,+}(x,t)&= \frac{m_x}{\epsilon}\Bigl(\partial_t \tilde{d}(x,t) + \epsilon^a c_1 e^{c_1 t}  \Bigr) +\epsilon m_\delta \dot{\xi}(t)\\
\Delta u^{\epsilon}(x,t)&=\frac{m_x}{\epsilon} \Delta \tilde{d}(x,t) + \frac{m_{xx}}{\epsilon^2}|\nabla \tilde{d}(x,t)|^2.
\end{align*}
Here $m_x$ denotes the partial derivative of $m(x,t)$ with respect to $x$. Then rewrite the reaction term using (\ref{TW}):
\[
\epsilon^{-2} ( f(m)+ \epsilon \xi^\epsilon)= \epsilon^{-2} \Bigl(-m''- m' c(\epsilon \xi^\epsilon+\epsilon^\beta) -\epsilon^\beta  \Bigr).
\]
By properly arranging the terms one gets
\begin{align*}
 \mathcal{L}(u^+):= \partial_t u^{\epsilon,+}(x,t)-\Delta u^{\epsilon}(x,t) -\epsilon^{-2} ( f(m(x,t))+ \epsilon \xi^\epsilon(t))= I_1 +I_2 +I_3+\epsilon^{\beta-2},
\end{align*}
where 
\begin{align*}
I_1 &=\frac{m_x}{\epsilon}\Bigl(\partial_t \tilde{d}(x,t) + \epsilon^a c_1 e^{c_1 t}-\Delta \tilde{d}(x,t)   +\epsilon^{-1} c(\epsilon \xi^\epsilon(t) + \epsilon^\beta) \Bigr)\\
I_2 &=\epsilon m_\delta \dot{\xi}(t)\\
I_3 &= \frac{m_{xx}}{\epsilon^2}\Bigl(1-|\nabla \tilde{d}(x,t)|^2  \Bigr).
\end{align*}
Here the first term accounts for the boundary motion. The statement that this term is small essentially means that the surface evolves with normal velocity $V= (n-1)\kappa + \epsilon^{-1} c(\epsilon \xi^\epsilon+\epsilon^\beta)$. The second term corresponds to the change of wave profile due to the change of noise. It is here that we need the pathwise bound (\ref{noi2}) on the derivative of $\xi^\epsilon$ to control this term. The third term essentially vanishes because close to $\Sigma(t)$ the function $\tilde{d}$ coincides with $d$ and therefore $|\nabla d|^2=1$. Off the boundary the derivative $m_{xx}$ becomes exponentially small such that we also control this term. In the end this means that the correction term $\epsilon^{\beta-2}$ dominates the dynamic. Let us make these considerations rigorous:\\

By (\ref{noi2}) $I_2 \leq C$ for every $\epsilon$ smaller than $\epsilon_0(\omega)$. For $d(x,t) \leq \eta \quad \nabla d(x,t)=1$ and therefore $I_3$ vanishes for such $x$.  For $d(x,t) \geq \eta$ Lemma \ref{CHL} (ii) implies:
\[
\frac{m_{xx}}{\epsilon^2}\Bigl(1-|\nabla \tilde{d}|^2  \Bigr) \leq  \frac{2A}{\epsilon^2} e^{-C/\epsilon}\rightarrow 0.
\]
To bound $I_1$ consider points $x$ close to $\Sigma(t)$. For all other $x$ the reasoning is as for $I_3$. For $x$ with $\text{dist}(x,\Sigma^{\pm,\epsilon}(t))\leq \frac{1}{2N}$ the functions $d(x,t)$ and $\tilde{d}(x,t)$ coincide and one obtains 
\[
 \partial_t d(x,t)= \Delta d(y,t)+\epsilon^{-1}c(\epsilon \xi^\epsilon(t) + \epsilon^\beta),
\]
where as before $y$ is the unique point in $\Sigma(t)$ such that $d(x,t)=\text{dist}(x,y)$. Plugging this into $I_1$ gives
\[
 I_1=\frac{m_x}{\epsilon}\ \Bigl\{ \Delta d^{\epsilon}(y,t)- \Delta d^{\epsilon}(x,t) +   \epsilon^a c_1 e^{c_1 t}\Bigr\}.
\]
Here one uses the fact that all the principle curvatures $\kappa_i(t,y)$ of the $\Sigma^{\pm,\epsilon}(t)$ are bounded by $N$ to obtain
\begin{align*}
| \Delta d^{\epsilon}(y,t)- \Delta d^{\epsilon}(x,t)|&=\left| \sum_{i=1}^{n-1} \kappa_i(y,t) - \sum_{i=1}^{n-1} \frac{\kappa_i(y,t)}{1-d(x,t) \kappa_i(y,t)}\right| \\
&=\sum_{i=1}^{n-1} |\kappa_i(y,t)| \frac{ d(x,t) |\kappa_i(y,t)|}{1-d(x,t) |\kappa_i(y,t)|} \\
&\leq 4 N^2 d(x,t), 
\end{align*}
because $\sup_{x \in [0, \frac{1}{2}]} \partial_{x} \frac{x}{1-x}=4$. Plugging this in yields
\begin{align*}
|I_1| \leq m_x\left(\frac{\tilde{d}^{\pm}(x,t)\pm \epsilon^a e^{c_1 t}}{\epsilon}, \epsilon \xi^\epsilon \pm \epsilon^\beta  \right) \frac{4 N^2 d(x,t)+\epsilon^a c_1 e^{c_1 t}}{\epsilon}.
\end{align*}
Choosing $c_1$ larger than $N^2$ and using $\sup_x  x m_x < \infty$ one obtains $|I_1| \leq C$. Thus altogether if $\epsilon$ is small enough the term $\epsilon^{\beta-2}$ will dominate everything else and one obtains
\[
\mathcal{L}(u^+) \geq 0.
\]
On the boundary $\frac{\partial u_+}{\partial \nu} =0$ due to the definition of $\tilde{d}$. So a standart comparison principle gives the desired result. The inequality for $u_-$ is shown in a similar manner.
\end{proof}
To finish the proof of the main theorem one needs the following Lemma:
\begin{lemme}\label{lemme2}
 Fix any time interval $[0,T]$. Denote by $W^{\pm,\epsilon}$ the random functions $[0,T] \ni t  \mapsto \frac{1}{c_0}\int_0^t \epsilon^{-1}c(\epsilon \xi^\epsilon(s) \pm \epsilon^\beta)\mathrm{ds} $. Then $c_0 W^{\pm,\epsilon}$ converges almost surely to $t \mapsto c_0 W(t)$ in $C^{0, \alpha}([0,T])$ for every $\alpha < \frac{1}{2}$.
\end{lemme}
\begin{proof}
Consider only $W^{+,\epsilon}(t)$ the calculation for $W^{-,\epsilon}(t)$ being the same. Fix $\alpha< \frac{1}{2}$ and a $\vartheta$ with $\alpha< \vartheta < \frac{1}{2}$. Then for $\PP$-almost every $\omega$ there exists a random constant $C$ such that 
\[
\sup_{-1 \leq s < t \leq T}\frac{|W(s)-W(t)|}{|s-t|^\vartheta} \leq C.
\]
Assume that $\epsilon$ is small enough to ensure $\epsilon \xi^\epsilon(t)+\epsilon^\beta \in [-\delta_0,\delta_0]$. (Recall that $c$ is only defined on $[-\delta_0,\delta_0]$.) Using Taylor-formula and $c(0)=0$ one can write for every $t$:
\[
\epsilon^{-1}c(\epsilon \xi^\epsilon(t) + \epsilon^\beta)=  c'(0) (\xi^\epsilon(t) + \epsilon^{\beta-1}) + \frac{1}{2} c''(a(t)) \epsilon^{-1} \bigl(\epsilon \xi^\epsilon(t) + \epsilon^\beta \bigr)^2,
\]
for some   $a(t)$ verifying $|a(t)| \leq |\epsilon \xi^\epsilon(t) + \epsilon^\beta|$.
Therefore one can write
\begin{align*}
\| c_0 W - c_0 W^{+,\epsilon} \|_\infty \leq& \sup_{s \in [0,T]} |c_0 W(s) - c_0 \int_0^s  \xi^\epsilon(t)dt | + c_0 T \epsilon^{\beta-1} \\
&+ T \sup_{\delta \in [- \delta_0, \delta_0]} |c''(\delta)| \Bigl(\sup_{s\in[0,T]}  \epsilon (\xi^\epsilon(s))^2 + \epsilon^{2\beta-1}  \Bigr).
\end{align*}
Due to (\ref{noi1}) the last terms converge to zero almost surely.  Therefore it remains to consider the first term. Due to $\dot{W}^\varepsilon(s)=\xi^\varepsilon(s)$ one obtains:
\begin{align*}
\sup_{s \in [0,T]} |c_0 W(s) &- c_0 \int_0^s  \xi^\epsilon(t)dt | \leq c_0 \sup_{s \in [0,T]} | W(s) - W(s)^\epsilon| + c_0 |W(0)^\epsilon | \\
&= c_0 \sup_{s \in [0,T]} \left| \int_{-\epsilon^\gamma}^{\epsilon^\gamma}\Bigl(W(s)-W(s-t)\Bigr) \rho^\epsilon(t)dt \right|+ c_0 \left| \int_{-\epsilon^\gamma}^{\epsilon^\gamma}\Bigl(W(0)-W(t)\Bigr) \rho^\epsilon(t)dt \right|\\
&\leq 2 c_0 C \Bigl(    \epsilon^\gamma \Bigr)^{\vartheta}  \rightarrow 0.
\end{align*}
Consider now the H\"older-seminorm
\begin{align*}
& \sup_{0 \leq s < t \leq T}\frac{1}{(t-s)^\alpha} \Bigl|c_0 W(t)- c_0 W^{+,\epsilon}(t) - c_0 W(s) + c_0 W^{+,\epsilon}(s)\Bigr|\\
&\quad =  \sup_{0 \leq s < t \leq T}\frac{1}{(t-s)^\alpha} \Bigl|c_0 W(t)-c_0 W(s) -\int_s^t \epsilon^{-1}c(\epsilon \xi^\epsilon(s) + \epsilon^\beta)\mathrm{ds}  \Bigr|\\
&\quad \leq \sup_{0 \leq s < t \leq T}\frac{1}{(t-s)^\alpha}\left| c_0 W(t) -c_0 W(s) - c_0 W^\epsilon(t) + c_0 W^\epsilon(s) \right| \\
& \qquad + \sup_{0 \leq s < t \leq T}\frac{1}{(t-s)^\alpha}\left|\int_s^t c_0 \epsilon^{\beta-1}  +  \sup_{\delta \in [- \delta_0, \delta_0]} |c''(\delta)| \Bigl(  \epsilon (\xi^\epsilon(u))^2 + \epsilon^{2\beta-1}  \Bigr) du \right|.
\end{align*}
Again the second term converges to zero. For the first term one gets:
\begin{align*}
& \sup_{0 \leq s < t \leq T}\frac{1}{(t-s)^\alpha}|c_0 W(t) -c_0 W(s) - c_0 W^\epsilon(t) + c_0 W^\epsilon(s)|\\ 
& \quad \leq c_0 \sup_{0 \leq s < t \leq T}\frac{1}{(t-s)^\alpha} \int_{-\epsilon^\gamma}^{\epsilon^\gamma}  \bigl( W(t)-W(t-u) -W(s) + W(s-u) \bigr) \rho^\epsilon(u) du\\
& \quad \leq c_0  \sup_{0 \leq s < t \leq T} \frac{(2(W(t)-W(s)))^{\frac{\alpha}{\vartheta}}}{(t-s)^\alpha} \int_{-\epsilon^\gamma}^{\epsilon^\gamma}   \bigl(W(t)-W(t-u) -W(s) + W(s-u) \bigr)^{1-\frac{\alpha}{\vartheta}} \rho^\epsilon(u) du\\
& \quad \leq c_0 \left( 2 C \right)^{\frac{\alpha}{\gamma}} \left( 2C (2\epsilon^\gamma)^\vartheta \right)^{1 -\frac{\alpha}{\vartheta}}.
\end{align*}
This shows the desired convergence.
\end{proof}
\begin{proof} (of Theorem (\ref{MR}))
Chose the initial configurations $u^\epsilon_0$ such that $u^{\epsilon}(x,0) \leq u^\epsilon_0(x) \leq u^{\epsilon,+}(x,0)$.  Define the stopping time $\tau(\omega):= \inf_{\epsilon} T^{\pm}_{\epsilon,N}$. Remark that $\tau$ is almost surely positive due to the boundedness of the $\| W^{\pm,\epsilon}  \|_{C^{\alpha/2}}$ and the $C^{2,\alpha}$ convergence of $d^{\epsilon,\pm}$ to $d$.

\vskip2ex

Then by Lemma \ref{subsup} one has for all times $0 \leq t \leq T(\omega)$ that $u^{\epsilon}(x,t) \leq u^\epsilon(x,t) \leq u^{\epsilon,+}(x,t)$. So one gets:
\begin{align*}
 \| u^{\epsilon}(\cdot,t) - \chi_{\Sigma(t)}(\cdot,t) \|_{L^2} &\leq \| u^{\epsilon}(\cdot,t) - u^{\epsilon,+}(\cdot,t) \|_{L^2(D)}+\| u^{\epsilon,+}(\cdot,t) - \chi_{\Sigma^{\epsilon,+}}(t)(\cdot,t) \|_{L^2(D)}\\
& \qquad \qquad +\| \chi_{\Sigma^{\epsilon,+}}(t)(\cdot,t) - \chi_{\Sigma(t)}(\cdot,t) \|_{L^2(D)} \\
&\leq \| u^{\epsilon,-}(\cdot,t) - u^{\epsilon,+}(\cdot,t) \|_{L^2(D)}+\| u^{\epsilon,+}(\cdot,t) - \chi_{\Sigma^{\epsilon,+}}(t)(\cdot,t) \|_{L^2(D)}+\\
& \qquad \qquad \| \chi_{\Sigma^{\epsilon,+}(t)}(\cdot,t) - \chi_{\Sigma(t)}(\cdot,t) \|_{L^2(D)} \\
&\leq \| u^{\epsilon,-}(\cdot,t) - \chi_{\Sigma^{\epsilon,-}(t)}(\cdot,t) \|_{L^2(D)}+2\| u^{\epsilon,+}(\cdot,t) - \chi_{\Sigma^{\epsilon,+}(t)}(\cdot,t) \|_{L^2(D)} +\\
& \qquad 2\| \chi_{\Sigma^{\epsilon,+}(t)}(\cdot,t) - \chi_{\Sigma(t)}(\cdot,t) \|_{L^2(D)} +\| \chi_{\Sigma^{\epsilon,-}(t)}(\cdot,t) - \chi_{\Sigma(t)}(\cdot,t) \|_{L^2(D)}.
\end{align*}
The supremum in time of the first two terms converges to zero due to the definition of $u^{\epsilon,\pm}$. Consider $\| \chi_{\Sigma^{\epsilon,-}(t)}(\cdot) - \chi_{\Sigma(t}(\cdot) \|_{L^2(D)}= \int_{\mathcal{O}}\Bigl( \chi_{\Sigma^{\epsilon,-}(t)}(x) - \chi_{\Sigma(t)}(x)\Bigr) dx$. By Lemma \ref{lemme2} and by Theorem \ref{Thm2.1} the signed distance functions converge in $C^{2, \alpha}(\mathcal{O})$ uniformly in time and therefore this term converges to zero. The convergence of the term involving $\chi_{\Sigma^{\epsilon,-}(t)}$ can be seen in the same way.
\end{proof}

\end{document}